\documentclass[12pt,twoside]{amsart}
\usepackage{latexsym,amsmath,amsopn,amssymb,amsthm,amsfonts}
\usepackage[T2A]{fontenc}
\usepackage[cp1251]{inputenc}
\usepackage[russian,english]{babel}
\usepackage{graphicx}

\newtheorem{theorem}{Theorem}

\newtheorem{proposition}{Proposition}

\newtheorem{remark}{Remark}

\newtheorem{definition}{Definition}

\newtheorem{corollary}{Corollary}

\newtheorem{lemma}{Lemma}

\newcommand{\Lin}{\operatorname{Lin}}
\newcommand{\const}{\operatorname{const}}
\newcommand{\Ad}{\operatorname{Ad}}
\newcommand{\ad}{\operatorname{ad}}

\newcommand{\I}{\operatorname{I}}
\newcommand{\sgn}{\operatorname{sgn}}
\newcommand{\Exp}{\operatorname{Exp}}
\newcommand{\trace}{\operatorname{trace}}
\newcommand{\Mink}{\operatorname{Mink}}

\newcommand{\Arcch}{\operatorname{arcch}}

\begin{document}
\begin{flushleft}
UDK 519.46 + 514.763 + 512.81 + 519.9 + 517.911
\end{flushleft}
\begin{flushleft}
MSC 2010 22E30, 49J15, 53C17
\end{flushleft}

\title[group $SO_0(2,1)$]{(Locally) shortest arcs of special sub-Riemannian metric on the Lie group $SO_0(2,1)$}
\author{V.\,N.\,Berestovskii}
\thanks{The author was partially supported by the Russian Foundation for Basic Research (Grant 14-01-00068-а) and a grant of the Government of the Russian Federation for the State Support of Scientific Research (Agreement \No 14.B25.31.0029)}
\address{Sobolev Institute of Mathematics SD RAS}
\address{V.N.Berestovskii}
\email{vberestov@inbox.ru}
\maketitle
\maketitle {\small
\begin{quote}
\noindent{\sc Abstract.}
The author finds geodesics, shortest arcs, cut locus, and conjugate sets for left-invariant
sub-Riemannian metric on the Lie group $SO_0(2,1)$ under the condition that the metric is right-invariant relative to the
Lie subgroup $SO(2)\subset SO_0(2,1).$
\end{quote}}

{\small
\begin{quote}
\end{quote}}

\section*{Introduction}

In this paper we find geodesics and shortest arcs for left-invariant and $SO(2)$-right-invariant sub-Riemannian metric $d$ on the Lie-Lorentz group $SO_0(2,1),$ where $SO(2)\subset SO_0(2,1).$

One can give the following natural geometric description of the metric $d.$ The Lie group $SO_0(2,1)$ can be interpreted as an effective transitive group of all preserving orientation isometries of the Lobachevskii plane $L^2$ with constant Gaussian curvature $-1$ and hence as the space $L^2_1$ of unit tangent vectors on  $L^2$. The space $L^2_1$ admits natural Riemannian metric (scalar product) $g_1$ by Sasaki  (see \cite{Sas} or section $1K$ in Besse book \cite{Bes}). In addition, canonical projection $p: (L^2_1,g_1)\rightarrow L^2$ (or, which is equivalent, 
$p: SO_0(2,1)\rightarrow  SO_0(2,1)/SO(2)$) is a \textit{Riemannian submersion} \cite{Bes}. The metric $d$ is defined by totally nonholonomic distribution $D$ on $SO_0(2,1),$ which is orthogonal to fibers of submersion $p,$ and restriction of scalar product $g_1$ to $D.$

Moreover, canonical projection
\begin{equation}
\label{subm}
p: (SO_0(2,1),d)\rightarrow  L^2
\end{equation}
is a \textit{submetry} \cite{BG}, a natural generalization of Riemannian submersion. The distribution $D$ on $L^2_1$ is nothing other than the restriction to $L^2_1$ of horizontal distribution of Levi-Civita connection \cite{Bes} for $L^2.$ 

Geodesics and shortest arcs in $(SO_0(2,1),d)$ are found with the help of these ideas from article \cite{BerZ}, general methods of work  \cite{Ber1}, and Gauss-Bonnet theorem \cite{Pog} for $L^2.$

All results of the paper are supplied with complete proofs. Along with presentation of geodesics with origin at unit in the form of product of two 1-parameter subgroups and in explicit matrix form, we also use their geometric interpretation
given in section \ref{lor}.

\begin{center}
\section{Preliminaries}
\label{prel}
\end{center}

\textit{Pseudoeuclidean space $E^{n,1}$} or \textit{Minkowski space-time $\Mink^{n+1}$}, where $n+1\geq 2,$ is vector space $\mathbb{R}^{n+1}$ with
\textit{pseudoscalar product} $\{(t,x),(s,y)\}:=-ts+(x,y).$ Here $(x,y)=xy^T$ is the \textit{standard scalar product of vectors $x,y\in \mathbb{R}^n.$} \textit{The Lorentz group} $SO_0(n,1)$ is the connected component of unit in group $P(n,1)$ of all linear \textit{pseudoisometric} (i.e. preserving pseudoscalar product 
$\{\cdot,\cdot\}$) transformations of the space $\Mink^{n+1}.$

Evidently,
\begin{equation}
\label{ps}
\{(t,x),(s,y)\}=((-t,x),(s,y))=((t,x)I,(s,y))=(t,x)I(s,y)^T,
\end{equation}
where $I$ is the matrix of linear mapping $(t,x)\rightarrow (-t,x),$ i.e. \textit{time reversing operator}. It follows from formula (\ref{ps}) that $A\in P(n,1)$ if and only if 
$$(t,x)I(s,y)^T=(t,x)AI((s,y)A)^T=(t,x)AIA^T(s,y)^T\Leftrightarrow AIA^T=I$$
for any $(t,x),(s,y)\in \mathbb{R}^{n+1},$ which is equivalent to 
\begin{equation}
\label{cps}
AIA^TI=E_{n+1},
\end{equation}
where $E_{n+1}=e$ is unit $(n+1)\times(n+1)-$matrix (we used the identity 
$I^2=E_{n+1}$).

\begin{remark}
\label{dif}
Unlike isometry group of Euclidean space $E^{n+1},$ which has two connected components (consisting respectively of orthogonal matrices preserving or not preserving orientation of the space $E^{n+1}$) the group $P(n,1)$ has 4 connected components, because conditions of (not) preserving time direction and (not) preserving orientation of the space $E^{n,1}$ are mutually independent for matrices
from $P(n,1)$. The group $SO_0(n,1)$ consists of those elements in $P(n,1)$ which
simultaneously preserve time direction and orientation of the space $E^{n,1}$.
\end{remark}

Let $A=A(s),$ $-\varepsilon< s< \varepsilon,$ be a continuously differentiable curve in $SO_0(n,1)$ such that $A(0)=e,$ $A'(0)=a.$ Then differentiation of identity  (\ref{cps}) at $s=0$ gives equality
$$a+Ia^TI=0\Leftrightarrow Ia + a^TI=0\Leftrightarrow Ia + (Ia)^T=0.$$
Conversely it is not difficult to prove that if a matrix $а$ satisfies this condition then matrix exponent $\exp(sa)\in SO_0(n,1)$ for all $s\in \mathbb{R}.$
Consequently, since $I^2=E_{n+1}$, then Lie algebra $\frak{so}(n,1)$ of Lie groups
$P(n,1)$ and $SO_0(n,1)$ is defined by the following equality
\begin{equation}
\label{son}
\frak{so}(n,1)= I\cdot\frak{so}(n).
\end{equation}
As a corollary,
\begin{equation}
\label{T}
(\frak{so}(n,1))^T= \frak{so}(n,1).
\end{equation}

The Lie group $Sl(n)\subset Gl(n)$ of all real $(n\times n)-$matrices
with determinant 1 is a closed connected subgroup of Lie group $Gl(n)$ with
the Lie algebra
\begin{equation}
\label{sl}
\frak{sl}(n)=\{a\in \frak{gl}(n,\mathbb{R}): \trace(a)=\sum_{l=1}^n a_{ll}=0.\}
\end{equation}

We shall be interested in the case $n=2.$ In view of equality (\ref{son}) 
matrices
\begin{equation}
\label{abc1}
a=e_{12}+e_{21},\quad b=e_{13}+e_{31},\quad c=e_{32}-e_{23},
\end{equation}
where \textit{$e_{ij}$ is quadratic matrix which has 1 in $i$-th row and $j$-th column and 0 in all other places}, constitute a basis of the Lie algebra
$\frak{so}(2,1).$ In addition, taking into account that $[a,b]=ab-ba$ etc., it is easy to find
\begin{equation}
\label{abca1}
[a,b]=-c, \quad [b,c]=a,\quad [c,a]=b.
\end{equation}
Analogously, in view of (\ref{sl}), matrices
\begin{equation}
\label{abc2}
a'=\frac{1}{2}(e_{12}+e_{21}),\quad b'=\frac{1}{2}(e_{11}-e_{22}),\quad c'=\frac{1}{2}(e_{12}-e_{21})\in \frak{sl}(2)
\end{equation}
constitute a basis of the Lie algebra $\frak{so}(2,1).$ Moreover,
\begin{equation}
\label{abca2}
[a',b']=-c', \quad [b',c']=a',\quad [c',a']=b'.
\end{equation}

\begin{theorem}
\label{liso}
Linear map $l: \frak{sl}(2)\rightarrow \frak{so}(2,1)$ such that
\begin{equation}
\label{l}
l(a')=a,\quad l(b')=b,\quad l(c')=c,
\end{equation}
is an isomorphism of Lie algebras. Furthermore, formula
\begin{equation}
\label{L}
L(\exp(w))=\exp(l(w)), w\in \frak{sl}(2),
\end{equation}
correctly defines epimorphism of Lie groups $L: Sl(2)\rightarrow SO_0(2,1)$ with
the kernel $\ker L= \{\pm E_2\}.$
\end{theorem}

\begin{proof}
First statement is a corollary of equalities (\ref{abca1}), (\ref{abca2}). Second statement follows from the first one and the fact that both Lie groups $SO_0(2,1)$ and $Sl(2)/\{\pm E_2\}$ are realized as effective full groups of 
preserving orientation of Lobachevskii plane of constant sectional curvature $-1$ for two of its H.Poincare models with stabilizers $\exp(\mathbb{R}c)$ and
$\exp(\mathbb{R}c')/\{\pm E_2\}$, isomorphic to $SO(2)$ (see sections \ref{conf} and \ref{lor}). Now note only that
$$L(-E_2)=L(\exp(2\pi c'))=\exp(2\pi c)=e_{11}+\cos (2\pi)(e_{22}+e_{33})+\sin(2\pi)c=E_{3}.$$
\end{proof}

Let $G$ and $H$ be Lie groups with Lie algebras $\frak{g}$ and $\frak{h}$; 
$\phi: G\rightarrow H$ is a Lie groups homomorphism. Then
\begin{equation}
\label{fe}
\phi \circ \exp_{\frak{g}} = \exp_{\frak{h}}\circ d\phi_e,
\end{equation}
moreover,
\begin{equation}
\label{dfe}
d\phi_e : (\frak{g},[\cdot,\cdot]) \rightarrow (\frak{h},[\cdot,\cdot])
\end{equation}
is a Lie algebra homomorphism (see lemma 1.12 in \cite{Hel}). If 
$g_0\in G$ then $\I(g_0): G\rightarrow G,$ where $\I(g_0)(g)=g_0gg_0^{-1}$ is inner automorphism of the Lie group $G.$ Consequently, 
$\Ad(g_0):=d\I(g_0)_e\in Gl(\frak{g})$ is automorphism of the Lie algebra 
$\frak{g}$ and $d\Ad_e(v):=\ad(v):=[v,\cdot]$ for $v\in \frak{g}$ \cite{Hel}. Therefore, on the ground of formula (\ref{fe}),
\begin{equation}
\label{I}
\I(g_0)\circ \exp= \exp \circ \Ad(g_0),
\end{equation}
\begin{equation}
\label{ad}
\Ad(\exp_{\frak{g}}(v))= \exp_{\frak{gl}(\frak{g})}(\ad(v)), v \in \frak{g}.
\end{equation}

\textit{Later $\Lin(a,b)$ will denote linear span of vectors $a,b.$ As an auxiliary   tool we shall use standard scalar product $(\cdot,\cdot)$ on Lie algebra
$\frak{gl}(n)=\mathbb{R}^{n^2}$}.

In case of left-invariant sub-Riemannian metrics on Lie groups, every geodesic
is a left shift of some geodesic which starts at the unit. Thus later we shall consider only geodesics with unit origin. Theorem 5 in paper \cite{Ber1} implies
the following theorem.

\begin{theorem}
\label{general}
Let $G$ be a connected Lie subgroup of the Lie group $Gl(n)$ with the Lie algebra
$\frak{g},$ $D$ is totally nonholonomic left-invariant distribution on $G,$ a scalar product $\langle \cdot,\cdot\rangle$ on $D(e)$ is proportional to restriction of the scalar product $(\cdot,\cdot)$ (to $D(e)$).
Then parametrized by arclength normal geodesic (i.e. locally shortest arc) 
$\gamma=\gamma(t),$ $t\in (-a,a)\subset \mathbb{R},$ $\gamma(0)=e,$ on $(G,d)$ with left-invariant sub-Riemannian metric $d$, defined by distribution $D$ and scalar product $\langle\cdot,\cdot\rangle$ on $D(e),$ satisfies the system of ordinary differential equations 
\begin{equation}
\label{dxsp2}
\stackrel{\cdot}\gamma(t)=\gamma(t)u(t),\,\,u(t)\in D(e)\subset \frak{g},\,\,\langle u(t),u(t)\rangle\equiv 1,
\end{equation}
\begin{equation}
\label{vot1}
pr_{\frak{g}}([u(t)^T,u(t)]+[u(t)^T,v(t)])=\stackrel{\cdot}u(t)+\stackrel{\cdot}v(t),
\end{equation}
where $u=u(t),$ $v=v(t)\in \frak{g},$ $(v(t),D(e))\equiv 0,$ $t\in (a,b)\subset \mathbb{R},$ are some real-analytic vector functions.
\end{theorem}

Equations (\ref{dxsp2}), (\ref{vot1}) imply

\begin{corollary}
\label{glob}
Every parametrized by arclength geodesic in $(G,d)$ is a part of unique parametrized by arclength geodesic $\gamma=\gamma(t),$ $t\in \mathbb{R},$ in
$(G,d).$
\end{corollary}

\begin{center}
\section{Search of geodesics in $(SO_0(2,1),d)$}
\end{center}

\begin{theorem}
\label{main}
Let be given the basis (\ref{abc1}) of the Lie algebra $\frak{so}(2,1),$ $D(e)=\Lin(a,b),$ and scalar product $\langle\cdot,\cdot\rangle$ on $D(e)$ 
with orthonormal basis $a,b.$ Then left-invariant distribution $D$ on the Lie group 
$SO_0(2,1)$ with given $D(e)$ is totally nonholonomic and the pair 
$(D(e),\langle\cdot,\cdot\rangle)$ defines left-invariant sub-Riemannian metric $d$ on $SO_0(2,1).$ Moreover, any parametrized by arclength geodesic $\gamma=\gamma(t),$
$t\in \mathbb{R},$ in $SO_0(2,1)$ with condition $\gamma(0)=e$ is a product of two 1-parameter subgroups:
\begin{equation}
\label{sol}
\gamma(t)=\exp(t(\cos\phi_0a + \sin \phi_0b - \beta c)) \exp(t\beta c),
\end{equation}
where $\phi_0,$ $\beta$ are some arbitrary constants.
\end{theorem}

\begin{proof}
The first statement of theorem follows from formula (\ref{abca1}).

It is clear that on $D(e)$
\begin{equation}
\label{scal}
\langle\cdot,\cdot\rangle =\frac{1}{2}(\cdot,\cdot).
\end{equation}
In consequence of theorem 3 in \cite{Ber1} every geodesic on 3-dimensional Lie group with left-invariant sub-Riemannian metric is normal. Then it follows from theorem \ref{general} that one can apply ODE
(\ref{dxsp2}),(\ref{vot1}) to find geodesics $\gamma=\gamma(t), t\in \mathbb{R},$ in $(SO(3),d)$.

It is clear that
\begin{equation}
\label{the}
u(t)=\cos\phi(t)a +\sin\phi(t)b,\quad v(t)=\beta(t)c,
\end{equation}
and the identity (\ref{vot1}) is written in the form
$$[\cos\phi(t)a +\sin\phi(t)b, \beta(t)c]=
\stackrel{\cdot}\phi(t)(-\sin\phi(t)a +\cos\phi(t)b)+
\stackrel{\cdot}\beta(t)c.$$
In consequence of (\ref{abca1}), expression in the left part of equality is equal to 
$$-\beta(t)(\cos\phi(t)b-\sin\phi(t)a).$$
We get identities $\stackrel{\cdot}\beta(t)=0,$
$\stackrel{\cdot}\phi(t)=-\beta(t).$ Hence
\begin{equation}
\label{cond}
\beta=\beta(t)=\const,\quad \phi(t)=\phi_0 -\beta t.
\end{equation}

In view of (\ref{dxsp2}), (\ref{the}), and (\ref{cond}), it must be
\begin{equation}
\label{dx}
\stackrel{\cdot}\gamma(t)=\gamma(t)(\cos (\phi_0- \beta t)a + \sin (\phi_0 - \beta t)b).
\end{equation}
Let us prove that (\ref{sol}) is a solution of ODE (\ref{dx}).
One can easily deduce from formulae (\ref{abca1}), (\ref{abc1}) equalities
\begin{equation}
\label{adm}
(\ad (c))=Ia,\quad (\ad (b))=b, \quad (\ad (a))=-(e_{23}+e_{32}),
\end{equation}
where $(f)$ denotes the matrix of linear map 
$f: \frak{so}(2,1)\rightarrow \frak{so}(2,1)$ in the base $a,b,c;$ later $(f)$ is identified with $f$. On the ground of formulae (\ref{ad}), (\ref{adm}), (\ref{cond}), (\ref{the})
$$\stackrel{\cdot}\gamma(t)=\exp(t(\cos\phi_0a + \sin \phi_0b - \beta c))(\cos\phi_0a + \sin \phi_0b - \beta c)
\exp(t\beta c)+$$
$$\gamma(t)(\beta c)=\gamma(t)\exp(-t\beta c)(\cos\phi_0a + \sin \phi_0b - \beta c)\exp(t\beta c)
+\gamma(t)(\beta c)=$$
$$\gamma(t)\exp(-t\beta c)(\cos\phi_0a + \sin \phi_0b)\exp(t\beta c)+\gamma(t)(-\beta c)+\gamma(t)(\beta c)=$$
$$\gamma(t)\cdot[\Ad(\exp(-t\beta c))(\cos\phi_0a + \sin \phi_0b)]=
\gamma(t)\cdot[\exp(\ad(-t\beta c))(\cos\phi_0a + \sin \phi_0b)]=$$
$$\gamma(t)\cdot[\exp(-t\beta(\ad(c)))(\cos\phi_0a + \sin \phi_0b)]=
\gamma(t)\cdot[\exp(-t\beta (Ia))(\cos\phi_0a + \sin \phi_0b)]=$$
$$\gamma(t)\cdot (\cos (\phi_0- \beta t)a + \sin (\phi_0-\beta t)b)=\gamma(t)u(t).$$
\end{proof}

\begin{remark}
Both 1-parameter subgroups from formula (\ref{sol}) are nowhere tangent to distribution $D$ for $\beta\neq 0$ so that any their interval has infinite length in metric $d.$
\end{remark}

\begin{remark}
\label{change}
To change a sign of $\beta$ in (\ref{sol}) is the same as to change a sign of $t$ 
and to change angle $\phi_0$ by angle $\phi_0\pm \pi.$
\end{remark}

\begin{remark}
\label{Ad}
For any matrix $B\in SO(2)=\exp(\mathbb{R}c),$ the map $l_B\circ r_{B^{-1}}$,
where $l_B$ is multiplication on the left by $B$, $r_{B^{-1}}$ is multiplication on the left by $B^{-1}$, is simultaneously automorphism $\Ad B$ of the Lie algebra 
$(\frak{so}(2,1),[\cdot,\cdot]),$ preserving $\langle\cdot,\cdot\rangle,$ and automorphism of the Lie group $SO_0(2,1),$ preserving distribution $D$ and metric
$d.$ In particular in view of (\ref{ad}),
(\ref{adm})
$$\Ad B (a -\beta c)=\exp(\phi_0 (Ia))(a -\beta c)=\cos\phi_0a + \sin \phi_0b - \beta c,$$ if
\begin{equation}
\label{AA}
B=\exp(\phi_0c)=\left(\begin{array}{ccc}
1 & 0 & 0\\
0 & \cos \phi_0 & -\sin \phi_0 \\
0 &   \sin \phi_0  & \cos \phi_0
\end{array}\right).
\end{equation}
\end{remark}

\begin{proposition}
\label{maint}
Let $\gamma(t),$ $t\in \mathbb{R},$ be geodesic in $(SO_0(2,1),d)$ defined by formula (\ref{sol}). Then for any $t_0\in \mathbb{R},$
\begin{equation}
\label{sol1}
\gamma(t_0)^{-1}\gamma(t)=\exp((t-t_0)(\cos(\beta t_0+\phi_0)a + \sin(\beta t_0+\phi_0)b -\beta c))
\exp((t-t_0)\beta c).
\end{equation}
\end{proposition}

\begin{proof}
On the basis of formulae (\ref{I}), (\ref{ad}), (\ref{adm}),
$$\gamma(t_0)^{-1}\gamma(t)=\exp(-t_0\beta c)\exp(-t_0(\cos\phi_0a + \sin\phi_0b -\beta c))\cdot$$
$$\exp(t(\cos\phi_0a + \sin \phi_0b -\beta c)) \exp(t\beta c)=$$
$$\exp(-t_0\beta c)\exp((t-t_0)(\cos\phi_0a + \sin \phi_0b -\beta c))\exp(t_0\beta c) \exp((t-t_0)\beta c)=$$
$$[\I(\exp(-t_0\beta c))(\exp((t-t_0)(\cos\phi_0a + \sin \phi_0b -\beta c)))]\cdot\exp((t-t_0)\beta c)=$$
$$\exp[\Ad(\exp(-t_0\beta c)((t-t_0)(\cos\phi_0a + \sin \phi_0b - \beta c))]\cdot\exp((t-t_0)\beta c)=$$
$$\exp[\exp(\ad(-t_0\beta c))((t-t_0)(\cos\phi_0a + \sin \phi_0b -\beta c))]\cdot\exp((t-t_0)\beta c)=$$
$$\exp[\exp(-t_0\beta (Ia))((t-t_0)(\cos\phi_0a + \sin \phi_0b - \beta c))]\cdot\exp((t-t_0)\beta c)=$$
$$\exp((t-t_0)(\cos(\phi_0-\beta t_0)a + \sin(\phi_0- \beta t_0)b - \beta c))\cdot\exp((t-t_0)\beta c).$$
\end{proof}

\begin{lemma}
\label{ee}
Let $x=(x_{ij})\in \frak{so}(2,1),$
\begin{equation}
\label{q}
q:= x_{21}^2 + x_{31}^2 - x_{32}^2, \quad \alpha:= \sqrt{|q|}.
\end{equation}
Then
\begin{equation}
\label{expo1}
\exp(x)=e+x + \frac{x^2}{2},\quad\mbox{if}\quad q = 0,
\end{equation}
\begin{equation}
\label{expo2}
\exp(x)=e+\frac{\sin \alpha}{\alpha}x+\frac{1-\cos \alpha}{\alpha^2}x^2,\quad\mbox{if}\quad q < 0,
\end{equation}
\begin{equation}
\label{expo3}
\exp(x)=e+\frac{\sh \alpha}{\alpha}x+\frac{\ch\alpha-1}{\alpha^2}x^2,\quad\mbox{if}\quad q> 0,
\end{equation}
where $e$ is unit matrix of third order.
\end{lemma}

\begin{proof}
Note that characteristic polynomial of the matrix $x$ is equal to
$$P(\lambda)=|x-\lambda e|=\left|\begin{array}{ccc}
-\lambda & x_{21} & x_{31}\\
x_{21} & -\lambda  & - x_{32}  \\
x_{31} & x_{32} & -\lambda
\end{array}\right|=-\lambda^3+ \lambda q,
$$
where $q$ is defined by formula (\ref{q}).

By Hamilton -- Cayley theorem \cite{Gant}, the matrix $x$ is a root of the polynomial $P(\lambda),$ i.e.
$x^3=q x.$ It follows from here that (\ref{expo1}) and
$$x^{2n+1}=(-1)^{n}\alpha^{2n}x,\quad x^{2n}=(-1)^{n+1}\alpha^{2n-2}x^2,\quad\mbox{if}\quad q < 0,\quad n\geq 1,$$
$$x^{2n+1}=\alpha^{2n}x,\quad x^{2n}=\alpha^{2n-2}x^2,\quad\mbox{if}\quad q > 0,\quad n\geq 1.$$
Therefore for $q < 0,$
$$\exp(x)=e+\sum_{n=1}^{\infty}\frac{x^n}{n!}=e + \frac{x}{\alpha}\sum_{n=0}^{\infty}\frac{(-1)^n\alpha^{2n+1}}{(2n+1)!}
-\frac{x^2}{\alpha^2}\sum_{n=1}^{\infty}\frac{(-1)^n\alpha^{2n}}{(2n)!},$$
and (\ref{expo2}) is fulfilled. Analogously for $q > 0,$
$$\exp(x)=e+\sum_{n=1}^{\infty}\frac{x^n}{n!}=e + \frac{x}{\alpha}\sum_{n=0}^{\infty}\frac{\alpha^{2n+1}}{(2n+1)!}
+\frac{x^2}{\alpha^2}\sum_{n=1}^{\infty}\frac{\alpha^{2n}}{(2n)!}$$
and (\ref{expo3}) is true.
\end{proof}

\begin{theorem}
\label{matrexp}
Let
\begin{equation}\label{mn1}
m=t,\,\,n=\frac{t^2}{2},\quad\mbox{if} \quad \beta^2 = 1,
\end{equation}
\begin{equation}\label{mn2}
m=\frac{\sin(t\sqrt{\beta^2-1})}{\sqrt{\beta^2-1}},\,\,n=\frac{1-\cos(t\sqrt{\beta^2-1})}{\beta^2-1},\quad\mbox{if}
\quad \beta^2 > 1,
\end{equation}
\begin{equation}\label{mn3}
m=\frac{\sh{(t\sqrt{1-\beta^2})}}{\sqrt{1-\beta^2}},\,\,n=\frac{\ch(t\sqrt{1-\beta^2})-1}{1-\beta^2},\quad\mbox{if}
\quad \beta^2 < 1,
\end{equation}
Then the geodesic $\gamma=\gamma(t)$ of left-invariant sub-Riemannian metric $d$ on the Lie group $SO_0(2,1)$ (see \,theorem~\ref{main}) is equal to 
\begin{equation}
\label{geod}
\tiny{\left(\begin{array}{ccc}
1+n & m\cos(\beta t-\phi_0)+\beta n\sin(\beta t-\phi_0) & \beta n\cos(\beta t-\phi_0) -m\sin(\beta t-\phi_0)\\
m\cos\phi_0+\beta n\sin\phi_0 & n\cos(\beta t-\phi_0)\cos\phi_0+\beta m\sin\beta t+(1-\beta^2n)\cos{\beta t}&
-n\sin(\beta t-\phi_0)\cos\phi_0+\beta m\cos\beta t - (1-\beta^2n)\sin\beta t \\
m\sin\phi_0-\beta n\cos\phi_0 & n\cos(\beta t-\phi_0)\sin\phi_0-\beta m\cos\beta t+(1-\beta^2n)\sin\beta t &
-n\sin(\beta t-\phi_0)\sin\phi_0+\beta m\sin\beta t+(1-\beta^2n)\cos{\beta t}
\end{array}\right).}
\end{equation}
\end{theorem}

\begin{proof}
Let $\phi_0=0$. Then (\ref{sol}) takes the form
$$\gamma(t)\mid_{\phi_0=0}=\exp(t(a-\beta c))\exp(t\beta c).$$
Using lemma~\ref{ee}, we get
$$\exp(t(a-\beta c))=\exp\left(t\left(\begin{array}{ccc}
0 & 1 & 0 \\
1 & 0 & \beta \\
0 & -\beta & 0
\end{array}\right)\right)=\left(\begin{array}{ccc}
1 & 0 & 0 \\
0 & 1 & 0 \\
0 & 0 & 1
\end{array}\right)+$$
$$m\left(\begin{array}{ccc}
0 & 1 & 0 \\
1 & 0 & \beta \\
0 & -\beta & 0
\end{array}\right)+n\left(\begin{array}{ccc}
0 & 1 & 0 \\
1 & 0 & \beta \\
0 & -\beta & 0
\end{array}\right)^2=
\left(\begin{array}{ccc}
1+ n & m & n\beta \\
m & 1+n(1-\beta^2) & m\beta \\
-n\beta & -m\beta & 1-n\beta^2
\end{array}\right).$$
By (\ref{AA}), matrices $B=\exp(\phi_0)$ and $\exp{(t\beta c)}$ commute. 
It follows from here, (\ref{sol}), and remark \ref{Ad} that
$$\gamma(t)= B\cdot\gamma(t)\mid_{\phi_0=0}\cdot B^{-1}=B\exp(t(a-\beta c))B^{-1}\exp(t\beta c)=$$
$$\left(\begin{array}{ccc}
1 & 0 & 0\\
0 & \cos \phi_0 & -\sin \phi_0 \\
0 &   \sin \phi_0  & \cos \phi_0
\end{array}\right)\left(\begin{array}{ccc}
1+ n & m & n\beta \\
m & 1+n(1-\beta^2) & m\beta \\
-n\beta & -m\beta & 1-n\beta^2
\end{array}\right)\times $$
$$\left(\begin{array}{ccc}
1 & 0 & 0\\
0 & \cos (\beta t-\phi_0) & -\sin (\beta t-\phi_0) \\
0 &   \sin (\beta t -\phi_0)  & \cos (\beta t-\phi_0)
\end{array}\right)=$$
$$\left(\begin{array}{ccc}
1+n & m & \beta n\\
m\cos\phi_0+\beta n\sin\phi_0 & (1+(1-\beta^2)n)\cos\phi_0+\beta m\sin\phi_0 & \beta m\cos\phi_0-(1-\beta^2n)\sin\phi_0 \\
m\sin\phi_0- \beta n\cos\phi_0 & (1+(1-\beta^2)n)\sin\phi_0-\beta m\cos\phi_0  & \beta m\sin\phi_0+(1-\beta^2n)\cos\phi_0
\end{array}\right)$$
$$\left(\begin{array}{ccc}
1 & 0 & 0\\
0 & \cos (\beta t-\phi_0) & -\sin (\beta t-\phi_0) \\
0 &   \sin (\beta t -\phi_0)  & \cos (\beta t-\phi_0)
\end{array}\right).$$

Calculation of the product of last two matrices finishes the proof of theorem~\ref{matrexp}.
\end{proof}

\begin{corollary}
\label{zero}
If $\phi_0=0$ then in notation (\ref{mn1}),(\ref{mn2}), and (\ref{mn3}),
\begin{equation}
\label{geodo}
\gamma(t)=
\left(\begin{array}{ccc}
1+n & m\cos{\beta t}+\beta n\sin{\beta t} & \beta n\cos{\beta t} -m\sin{\beta t}\\
m & \beta m\sin{\beta t}+(1+(1-\beta^2)n)\cos{\beta t}&
\beta m\cos{\beta t} - (1+(1-\beta^2)n)\sin{\beta t} \\
-\beta n & -\beta m\cos{\beta t}+(1-\beta^2n)\sin{\beta t} &
\beta m\sin{\beta t}+(1-\beta^2n)\cos{\beta t}
\end{array}\right).
\end{equation}
\end{corollary}

\section{Conformal Poincare model of Lobachevskii plane}
\label{conf}

Let us recall a known geometric interpretation of the Lie group $Sl(2)/\{\pm E_2\}$ as the group of all preserving orientation isometries of Lobachevskii plane $L^2.$ Upper half-plane $P=\{z=x+yi: y>0\}$ of complex plane $\mathbb{C}$ with metric tensor
$ds^2=\frac{dx^2+dy^2}{y^2}$ is \textit{conformal Poincare model of Lobachevskii plane} $L^2.$ The group $Sl(2)$ acts by preserving orientation isometries on the half-plane $P$ by means of real linear-fractional transformations 
$$z\rightarrow \frac{az+b}{cz+d}; \quad
\left(\begin{array}{cc}
a & b\\
c & d
\end{array}\right)\in Sl(2).$$
It is clear that the kernel of this action is the central subgroup 
$\{\pm E_2\}\subset Sl(2).$ Consequently the Lie group $Sl(2)/\{\pm E_2\}$ is really the group of (all) preserving orientation isometries of the Lobachevskii plane $L^2.$

One can easily check that Lie subgroup $SO(2)/\{\pm E_2\}\subset Sl(2)/\{\pm E_2\}$ is the stabilizer of point $z_0=i$ relative to indicated action, i.e. consists exactly of those elements of the group $Sl(2)/\{\pm E_2\},$ which fix this point; moreover the group $SO(2)/\{\pm E_2\}$ acts (simply) transitively by rotations on the circle, the set of unit tangent vectors to the half-plane $P$ at the point $i$. Therefore $P$ is naturally identified with the quotient-space 
$(Sl(2)/\{\pm E_2\})/(SO(2)/\{\pm E_2\}).$

\section{The group $SO_0(n,1)$ is a connected isometry group of the space $L^n$}
\label{lor}

In section \ref{prel} the group $SO_0(n,1)$ acted by pseudoisometries from the right on vector-rows of Minkowski space-time $\Mink^{n+1}$. In consequence of formula (\ref{T}) we can, and shall suppose that the group $SO_0(n,1)$ acts from the left on vector-columns of Minkowski space-time $\Mink^{n+1}$.

Orbit $SO_0(n,1)\cdot (w_0=(1,0,...,0)^T)$ of event $(1,0,...,0)^T\in \Mink^{n+1}$ is the upper (more exactly, for "usual" disposition of coordinate axes, "right") sheet of two-sheeted hyperboloid 
\begin{equation}
\label{hyper}
-t^2+\sum_{k=1}^nx_k^2=\{(t,x)^T,(t,x)^T\}=-1,\quad t>0.
\end{equation}
Restriction of pseudoscalar product $\{\cdot,\cdot\}$ to the tangent vector bundle of this orbit is scalar product and the orbit with this scalar product is isometric to $n$-dimensional Lobachevskii space $L^n$ of constant sectional curvature $-1$. Therefore later $L^n$ will denote this orbit. Then $SO_0(n,1)$ is automatically the largest connected transitive isometry group of the space $L^n.$ Moreover, subgroup 
$$SO(n):=\left(\begin{array}{cc}
1 & 0\\
0 & SO(n)
\end{array}\right)\subset SO_0(n,1)$$
is stabilizer of the group $SO_0(n,1)$ at the point $w_0.$ Therefore $L^n$ is naturally identified with homogeneous space $SO_0(n,1)/SO(n)$ while the action 
of the group $SO_0(n,1)$ on  $L^n$ is identified with its standard left action on а
$SO_0(n,1)/SO(n).$ Canonical projection 
$p: SO_0(n,1)\rightarrow L^n=SO_0(n,1)/SO(n)$ is defined by formula 
$p(g)=gSO(n).$

\begin{remark}
\label{sect} Unlike conformal model of Lobachevskii plane from section \ref{conf} geodesics, equidistant curves, circles, and horocycles of Lobachevskii plane have 
simple visual description in "relativistic" model from section \ref{lor}. Namely, their collection for $n=2$ is the set of all sections of the sheet (\ref{hyper}) of two-sheeted hyperboloid by planes. Moreover all tangent vectors of any such plane
are pseudoorthogonal to some non-zero vector $v$: to space-like ($\{v,v\}> 0$) for geodesics and equidistant curves (and in the case of geodesic, corresponding plane passes through origin of coordinates $O$), time-like ($\{v,v\}< 0$) for circles, and isotropic $\{v,v\}=0$) for horocycles. Analogous statements are true for $n > 2.$
\end{remark}

The Lorentz group $SO_0(n,1)$ is diffeomorphic to the space $L^n_1$ of all unit tangent vectors to $L^n.$ Namely, any element $g\in SO_0(n,1)$ corresponds to unit tangent vector $f(g):=g(v_0)$ to $L^n$ at point $g(w_0),$ where $v_0=(0,1,0,\dots,0)^T$ is unit tangent vector to $L^n$ at the point $w_0.$

Next statements of this section are based on information given in the introduction.

For any geodesic path
$\gamma(t)$, $0\leq t \leq t_1,$ in $(SO_0(2,1),d)$ with arbitrary origin 
$g\in SO_0(2,1),$ $f(\gamma(t))$, $0\leq t \leq t_1,$ is a parallel vector field (in \textit{Lobachevskii plane!}) along projection $p(\gamma(t))$, $0\leq t \leq t_1,$ in the sense of  \cite{Pog} with initial unit tangent vector $f(\gamma(0))=g(v_0)\in L^2_1$ \cite{BerZ}.

1) In particular, if $\gamma(t),$ $t\in \mathbb{R},$ is geodesic in $(SO_0(2,1),d)$ of the form (\ref{sol}) with $\phi_0=0$, then $\gamma'(0)=v_0$ and $f(\gamma(t))$ $0\leq t \leq t_1,$ is parallel vector field in $L^2$ along $p(\gamma(t))$,
$t\in \mathbb{R},$ with initial unit tangent vector $\gamma'(0).$

2) Canonical projection $p: (SO_0(2,1),d)\rightarrow L^2$ is submetry \cite{BG}, \cite{BerZ}.

3) If $\gamma(t),$ $0\leq t \leq t_1,$ is any (parametrized by arclength) geodesic in $(SO_0(2,1),d)$ then its projection $p(\gamma(t)),$ $0\leq t \leq t_1,$ in
$L^2$ is parametrized by arclength.

4) On the ground of proposition \ref{maint}, remark \ref{Ad}, and left invariance of the metric $d$, for the search of all shortest arcs in $(SO_0(2,1),d)$ it is enough to find all noncontinuable shortest arcs of the form $\gamma(t)$, 
$0\leq t \leq t_1,$ (\ref{sol}) with $\phi_0=0.$

\section{Shortest arcs in the Lie group $(SO_0(2,1),d)$}
\label{vuch}

Let us use statements 1) --- 4) from section \ref{lor} to find shortest arcs in
$(SO_0(2,1),d)$. In particular, it is sufficient to investigate segments of geodesics of the form 
\begin{equation}
\label{ge}
\gamma(t)=\exp(t(a-\beta c))\exp(t\beta c),\quad 0\leq t\leq t_1,
\end{equation}
and their projections
\begin{equation}
\label{pr}
x(t):= p(\gamma(t))=\gamma(t)\cdot w_0=\gamma(t)\cdot (1,0,0)^T=(1+n,m,-\beta n)^T,\quad 0\leq t\leq t_1,
\end{equation}
to the plane $L^2,$ where $m$, $n$ are defined by formulae (\ref{mn1}),(\ref{mn2}),(\ref{mn3}) (we used formula (\ref{geodo})).

Let us formulate the Gauss-Bonnet theorem \cite{Pog}. Let $M$ be
two-dimensional oriented manifold with Riemannian metric $ds^2,$ $\Phi$ is a region in $M,$ homeomorphic to disc and bounded by closed piece-wise regular curve 
$\gamma$ with regular links $\gamma_1,\dots, \gamma_n,$ forming angles 
$\alpha_1, \dots, \alpha_n$ from the side of region $\Phi.$ Direction on the curve
$\gamma$ is given so that the region $\Phi$ is situated from the right under bypass of the curve in this direction. Then

\begin{theorem}
\label{gauss}
$$\sum_{k=1}^n \int_{\gamma_k}\kappa ds + \sum_{k=1}^n (\pi-\alpha_k) = 2\pi -\int\int_{\Phi}Kd\sigma,$$
where $\kappa$ is geodesic curvature at points of links of the curve, $K$ is
Gaussian (sectional) curvature of the surface $(M,ds^2),$ and integration in the
right part of equality is taken by area element of the region $\Phi.$
\end{theorem}

Formula for calculation of $\kappa$ in semigeodesic system of coordinates is given in \cite{Pog}. Below we construct a semigeodesic system of coordinates 
$(u,v)$ in $L^2.$

Canonical parametrized by arclength geodesic in $L^2$ has a form 
$\tilde{\gamma}(s)=(\ch s,\sh s,0).$ It follows from here and the invariance of distance in $L^2$ with respect to action of the group $SO_0(2,1)$ that distance between arbitrary points $(t,x,y)^T$ and $(t_1,x_1,y_1)^T$ in $L^2$ is equal to
\begin{equation}
\label{dist}
\rho((t,x,y)^T,(t_1,x_1,y_1)^T) = \Arcch (-\{(t,x,y)^T,(t_1,x_1,y_1)^T\}).
\end{equation}

Let $(t,x,y)^T$ be arbitrary point in $L^2.$ Then $(1/\sqrt{t^2-y^2})(t,0,y)^T\in L^2$ and by (\ref{dist}), distances from this point to points $(t,x,y)^T$ and $w_0=(1,0,0)^T$ are equal respectively to $\Arcch(\sqrt{t^2-y^2})$ and 
$\Arcch(t/\sqrt{t^2-y^2}).$ In accordance with this, define coordinates $u,v$
of the point $(t,x,y)^T$ by formulae
\begin{equation}
\label{uv}
u = (\sgn x)\Arcch(\sqrt{t^2-y^2}), \quad v = (\sgn y)\Arcch\left(\frac{t}{\sqrt{t^2-y^2}}\right).
\end{equation}
Taking into account what we said above, it is not difficult to check that all points of line $u=u_0$ are disposed from line $u=0$ on distance $|u_0|,$ moreover line $v=v_0$ gives shortest junction of point $(u_0,v_0)$ with the line $u=0,$
and the line $u=0$ is geodesic, for which $v$ is parametrization by arclength. It follows from here that the length element in these coordinates is defined by formula
\begin{equation}
\label{ds2}
ds^2=du^2+\ch^2(u)dv^2,
\end{equation}
i.e. $(u,v)$ is semigeodesic system of coordinates in $L^2.$ Furthermore  \textit{first partial derivatives of components of metric tensor and Christoffel symbols in this system of coordinates are equal to zero on the line $u=0.$}
One can easily deduce from here, proposition \ref{maint}, formula (\ref{dx}),
and formula for $\kappa$ in \cite{Pog} the following proposition.

\begin{proposition}
\label{kappa}
Projection (\ref{pr}) has constant geodesic curvature $\kappa$ and 
$$u(0)=0,\quad v(0)=0,\quad u'(0)=1,\quad v'(0)=0,\quad \kappa = -v''(0).$$
for its coordinate presentation $(u(t),v(t))=(u(x(t)),v(x(t)))$.
\end{proposition}

\begin{remark}
N.I.Lobachevskii already constructed by other method a semigeodesic system of coordinates $(u,v)$ in $L^2$ with length element (\ref{ds2}) and wrote explicitly the formula (\ref{ds2}). He found an analogous system of coordinates and formula of length element for $L^3.$  This contradicts the conventional belief that Lobachevskii himself didn't present a model of his geometry (and therefore gave no
complete logical justification of this geometry).
\end{remark}

\begin{proposition}
\label{kap}
Geodesic curvature of the projection (\ref{pr}) is equal to $\kappa=\beta.$
\end{proposition}

\begin{proof}
In consequence of formulae (\ref{pr}), (\ref{mn1}), (\ref{mn2}), (\ref{mn3}), and (\ref{uv})
$$n(0)=m(0)=0,\quad n'=m,\quad n'(0)=0,\quad n''(0)=1$$
and for small positive $t$
$$v(t)=\sgn(-\beta)\Arcch\left(\frac{1+n}{\sqrt{1+2n+(1-\beta^2)n^2}}\right),$$
$$v'(t)= \frac{\sgn(-\beta)}{\sh(\Arcch((1+n)/\sqrt{1+2n+(1-\beta^2)n^2}))}\cdot
\left(\frac{1+n}{\sqrt{1+2n+(1-\beta^2)n^2}}\right)'=$$
$$\frac{\sqrt{1+2n+(1-\beta^2)n^2}}{|\beta|n}\cdot \frac{\sgn(-\beta)\beta^2nn'}{(1+2n+(1-\beta^2)n^2)^{3/2}}=
\frac{\sgn(-\beta)|\beta|n'}{1+2n+(1-\beta^2)n^2}.$$
On the ground of proposition \ref{kappa}
$$\kappa = -v''(0)=\sgn(\beta)|\beta|n''(0)=\beta.$$
\end{proof}

According to theorem \ref{matrexp} we shall consider later 4 cases:\\
 I) $\beta=0$, II) $0< \beta^2 < 1$, III) $\beta^2 = 1$, IV) $1< \beta^2.$

The next corollary follows immediately from proposition \ref{kap} and known facts of hyperbolic geometry.

\begin{corollary}
\label{geoc}
Projection (\ref{pr}) is I) geodesic, II) equidistant curve, \\III) horocycle, IV) circle.
\end{corollary}

\begin{lemma}
\label{zer}
In the case I) every segment $\gamma(t)$, $0\leq t \leq t_1,$ is a shortest arc.
\end{lemma}

\begin{proof}
On the basis of corollary \ref{geoc} every segment (\ref{pr}) is a shortest arc. Assume that the statement of lemma is false. Then there is another shortest
geodesic segment $\gamma_0(t)$, $0\leq t \leq t_0 < t_1,$ in
$(SO_0(2,1),d)$ with the same ends as for the segment in lemma. Hence in view of
3) the path $x_0(t):= p(\gamma_0(t))$, $0\leq t \leq t_0,$ has length $t_0 < t_1$ and the same ends as the shortest arc (\ref{pr}) with the length $t_1,$ a contradiction.
\end{proof}

\begin{proposition}
\label{area}
Suppose that projection (\ref{pr}) of geodesic segment (\ref{ge}), where
$\beta\neq 0,$ has no self-intersection (this is always true in cases II and III and it is true in the case IV when $0\leq t_1< 2\pi/\sqrt{\beta^2-1}$),
$S(t_1)=S(t_1,\beta)$ is area of curvilinear digon $P$ in $L^2$ bounded by the segment (\ref{pr}) and the shortest arc $[x(0)x(t_1)]$ with the length $r=r(t_1)$ in $L^2,$ $\psi=\psi(t_1,\beta)$ is the angle of the digon $P$. Then
\begin{equation}
\label{exist}
S(t_1)=|\beta|t_1-2\psi,\quad r= \Arcch((1+n)(t_1)), \quad r'(t_1)=\cos\psi = \frac{m}{\sqrt{n(n+2)}}.
\end{equation}
In addition $S'(t_1)>0,$ $t_1>0;$ $0< \psi < \pi/2$ in cases II and III, and in the case IV when $0\leq t_1< \pi/\sqrt{\beta^2-1}.$
\end{proposition}

\begin{proof}
In consequence of remark \ref{change} one can assume that $\beta > 0.$ Segment
$[x(0)x(t_1)]$ has geodesic curvature $0.$ Then the first equality in  (\ref{exist}) is a direct corollary of proposition \ref{kap} and theorem  \ref{gauss} for $K=-1$, the second equality follows from formulae (\ref{pr}) and (\ref{dist}), the next one is a well-known statement of Riemannian geometry
(on existence of strong angle), the last equality is result of differentiation of 
second equality in (\ref{exist}). Inequalities $0 < \psi(t_1)< \pi/2$ are valid in indicated cases in view of the last equality in (\ref{exist}), first formulae in  (\ref{mn1}), (\ref{mn2}), (\ref{mn3}) and equality $\lim_{t_1\rightarrow +0}\psi(t_1)=0.$

Let us prove the rest of the statement. It is known that in $L^2$
\begin{equation}
\label{l1}
l(r,\alpha)= \alpha \sh r,
\end{equation}
\begin{equation}
\label{S}
S(r,\alpha)=\int_0^r \alpha \sh s ds = \alpha \ch s|_0^r= \alpha (\ch r-1),
\end{equation}
where $l(r,\alpha)$ is the length of arc of circle of radius $r$ with central angle $\alpha \leq 2\pi,$ and $S(r,\alpha)$ is area of corresponding sector. From here,  (\ref{exist}), and second formulae in (\ref{mn1}), (\ref{mn2}), (\ref{mn3})
follow relations
$$S'(t_1)=(\ch r-1)\psi'(t_1)=|\beta|-2\psi'(t_1),$$
\begin{equation}
\label{dpsi}
\psi'(t_1)=\frac{|\beta|}{\ch r+1}=\frac{|\beta|}{n+2},
\end{equation}
\begin{equation}
\label{der}
\quad S'(t_1)=\frac{|\beta|n}{n+2}(t_1)>0, \quad t_1>0.
\end{equation}
\end{proof}

\begin{proposition}
\label{zer1}
1)If $\beta\neq 0$ then geodesic segment (\ref{ge}) is noncontinuable shortest arc 
when its projection (\ref{pr}) is a) one time passing circle $C$ bounding disc
with area $S(t_1)\leq \pi$ or b) curve without self-intersections bounding together with the shortest arc $[x(0)x(t_1)]$ in $L^2$ digon $P$ in $L^2$ with area
$S(t_1) =\pi$.

2) For every $\beta\neq 0$ there is unique $t_1 > 0$ such that the condition 
a) or b) is satisfied; a) is satisfied only if $|\beta|\geq 3/\sqrt{5}.$
\end{proposition}

\begin{proof}
1) a) It is clear that $\gamma(t_1)\in SO(2).$ Then in consequence of remark  \ref{Ad} for the same $\beta$ and any $\phi_0$, segment of geodesic 
(\ref{sol}) under $t\in [0,t_1]$ joins the same points as (\ref{ge}). Consequently every continuation of the segment (\ref{ge}) is not a shortest arc.

Let us suppose that there exists a shortest arc $\gamma_2(t),$ 
$0\leq t\leq t_2 < t_1,$ in $(SO_0(2,1),d)$ which joins points $\gamma(0)=e$ and
$\gamma(t_1).$ Then projection $x_2(t)=p(\gamma_2(t)),$ $0\leq t\leq t_2,$ is one time passing circle $C_2$ in $L^2$ with length $t_2 < t_1$ and therefore bounds
a disc with area $S(t_2)< S(t_1)\leq \pi.$ Consequently on the ground of the Gauss-Bonnet theorem results of parallel translations of nonzero vectors along $C$ and $C_2$ in $L^2$ are different. Then $\gamma_2(t_2)\neq \gamma(t_1)$ 
in view of geometric interpretation of geodesics in $(SO_0(2,1),d)$ given in section \ref{lor}, a contradiction.

b) Let $P'$ be a digon, symmetric to the digon $P$ relative to segment $x(0)x(t_1).$ Since $S(t_1)=\pi$ then by the Gauss-Bonnet theorem results of parallel translations in $L^2$ of tangent vectors along closed paths, bounding $P$ and $P',$ are equal. Therefore on the ground of remarks \ref{change}, \ref{Ad} and geometric interpretation of geodesics in $(SO_0(2,1),d),$ given in section \ref{lor},
a curve in $L^2,$ symmetric to the projection (\ref{pr}) of segment (\ref{ge}) relative to segment $x(0)x(t_1),$ is presented in the form $p(\gamma_1(t)),$ 
$0\leq t \leq t_1,$ where $\gamma_1$ is geodesic in $(SO_0(2,1),d)$ such that  
$\gamma_1(0)=\gamma(0),$ $\gamma_1(t_1)=\gamma(t_1).$ Consequently every continuation of the segment (\ref{ge}) is not a shortest arc.

Let us suppose that there is a shortest arc $\gamma_2(t),$ $0\leq t\leq t_2 < t_1,$ in $(SO_0(2,1),d),$ joining points $\gamma(0)=e$ and $\gamma(t_1).$ Then
in consequence of remarks \ref{change} and \ref{Ad} we can assume that curves
(\ref{pr}) and $x_2(t)=p(\gamma_2(t)),$ $0\leq t\leq t_2,$ lie on the one side
of the shortest arc $[x(0)x(t_1)]$ and join ends of this shortest arc. Consequently on the ground of proposition \ref{geoc} the digon $P$ and digon $P_2,$ bounded by the shortest arc $[x(0)x(t_1)]$ and the curve $x_2(t),$ $0\leq  t \leq t_2,$ are convex, moreover intersection of their boundaries is the shortest arc 
$[x(0)x(t_1)]$  because $t_2 < t_1.$ Therefore in view of last inequality the curve 
$x_2(t),$ $0< t < t_2,$ lies inside $P$ and $S(t_2)< S(t_1)=\pi,$ where $S(t_2)$ is area of the digon $P_2.$ Consequently on the ground of the Gauss-Bonnet theorem
results of parallel translations of nonzero tangent vectors along boundaries of $P$ and $P_2$ in $L^2$ are different. Then $\gamma_2(t_2)\neq \gamma(t_1)$
in view of geometric interpretation of geodesics in $(SO_0(2,1),d),$ given in section \ref{lor}, a contradiction.

2) On the basis of corollary \ref{geoc} and first equality in (\ref{exist}) the condition a) is fulfilled only if 
$$\beta^2 > 1,\quad S\left(\frac{2\pi}{\sqrt{\beta^2-1}}\right)= |\beta|\frac{2\pi}{\sqrt{\beta^2-1}} - 2\pi \leq \pi
\Leftrightarrow |\beta|\geq \frac{3}{\sqrt{5}}.$$
If $0< |\beta|< \frac{3}{\sqrt{5}}$ then in consequence of proposition \ref{area} there exists unique $t_1> 0$ for which the condition b) is satisfied.
\end{proof}

Later for every number $\beta\neq 0$ we shall find a number $t_1=t_1(\beta),$ satisfying conditions of proposition \ref{zer1}.

II) On the basis of first formula in (\ref{exist}) it must be
\begin{equation}
\label{ds}
\pi=|\beta|t_1-2\psi,\quad \frac{|\beta|t_1}{2}=\frac{\pi}{2}+\psi,\quad \frac{\pi}{2} < \frac{|\beta|t_1}{2} < \pi.
\end{equation}
We deduce from formulae (\ref{ds}), (\ref{mn3}) and last equality in (\ref{exist}) that
\begin{equation}
\label{2phis}
\sin \left(\frac{|\beta|t_1}{2}\right)=\cos \psi=
\frac{\sqrt{1-\beta^2}\ch(t_1\sqrt{1-\beta^2}/2)}{\sqrt{\ch^2(t_1\sqrt{1-\beta^2}/2)-\beta^2}},
\end{equation}
\begin{equation}
\label{2phic}
-\cos \left(\frac{|\beta|t_1}{2}\right) = \sin \psi =
\frac{|\beta|\sh (t_1\sqrt{1-\beta^2}/2)}{\sqrt{\ch^2(t_1\sqrt{1-\beta^2}/2)-\beta^2}},
\end{equation}

III) On the basis of (\ref{mn1}) and first formula in (\ref{exist}) we must have the same formulae (\ref{ds}) for $|\beta|=1$ and
\begin{equation}
\label{fin}
\sin\left(\frac{t_1}{2}\right)=\cos \psi=\frac{1}{\sqrt{1+(t_1/2)^2}}.
\end{equation}
\begin{equation}
\label{fin1}
-\cos\left(\frac{t_1}{2}\right)=\sin \psi =\frac{t_1/2}{\sqrt{1+(t_1/2)^2}}, \quad \frac{\pi}{2} < \frac{t_1}{2} < \pi.
\end{equation}

IV) The least positive number $t_0$ such that $x(t_0)=w_0$ is equal to
$2\pi/\sqrt{\beta^2-1}.$ In addition
$\lim_{t\rightarrow t_0-0}\psi(t)=\pi$ and $S(t_0)=(2\pi|\beta|/\sqrt{\beta^2-1})-2\pi$ by (\ref{exist}). $S(t_0)=2\pi$ (respectively $S(t_0)=\pi$) if $|\beta|=2/\sqrt{3}$ (respectively $|\beta|=3/\sqrt{5}$). Let us consider three cases:\\
a) $|\beta|\geq 3/\sqrt{5}, \quad$ b) $1< |\beta| \leq 2/\sqrt{3}, \quad,$ c) $2/\sqrt{3}< |\beta| < 3/\sqrt{5}.$

a) In this case in view of proposition \ref{zer1}
\begin{equation}
\label{ne2}
t_1=2\pi/\sqrt{\beta^2-1}.
\end{equation}

b) In this case there is a unique number $t_1$ such that 
$0< t_1 \leq \pi/\sqrt{\beta^2-1}$ and $S(t_1)=\pi.$ Then in consequence of
(\ref{mn2}) and proposition \ref{area}
\begin{equation}
\label{betat}
\pi = |\beta|t_1 -2 \psi,\quad 0< \psi \leq \pi/2,\quad \frac{\pi}{2}< \frac{|\beta|t_1}{2}=\frac{\pi}{2}+\psi \leq \pi.
\end{equation}
\begin{equation}
\label{nec}
\sin\left(\frac{|\beta|t_1}{2}\right)= \cos\psi=
\frac{\sqrt{\beta^2-1}\cos(t_1\sqrt{\beta^2-1}/2)}{\sqrt{\beta^2-\cos^2(t_1\sqrt{\beta^2-1}/2)}},
\end{equation}
\begin{equation}
\label{ne}
-\cos\left(\frac{|\beta|t_1}{2}\right)= \sin\psi=
\frac{|\beta|\sin(t_1\sqrt{\beta^2-1}/2)}{\sqrt{\beta^2-\cos^2(t_1\sqrt{\beta^2-1}/2)}}.
\end{equation}

c) We get the same formulae (\ref{nec}), (\ref{ne}) but in this case
\begin{equation}
\label{betat1}
\pi/\sqrt{\beta^2-1}< t_1 < 2\pi/\sqrt{\beta^2-1},\quad \pi/2< \psi <  \pi,\quad \pi<
\frac{|\beta|t_1}{2}=\frac{\pi}{2}+\psi < \frac{3\pi}{2}.
\end{equation}

\begin{theorem}
\label{mon}
Let (\ref{ge}), where $t_1=t_1(|\beta|),$ is a noncontinuable shortest arc in 
$(SO_0(2,1),d).$ Then $t_1(|\beta|)$ is stronly decreasing function of $|\beta|>0.$
\end{theorem}

\begin{proof}
In consequence of p. a) of proposition \ref{zer1} and formula (\ref{ne2}) this
statement is true for $|\beta|\geq 3/\sqrt{5}.$ If $0< |\beta| < 3/\sqrt{5}$ then
the second formula in (\ref{ds}) is valid. In consequence of it and (\ref{dpsi})
$$t_1+|\beta|\frac{dt_1}{d|\beta|}=2\psi'(t_1)\cdot\frac{dt_1}{d|\beta|}=\frac{2|\beta|}{n+2}\cdot\frac{dt_1}{d|\beta|},
\quad t_1=\frac{-|\beta|n}{n+2}\cdot\frac{dt_1}{d|\beta|}.$$
It follows from here and second formulae in (\ref{mn1}), (\ref{mn2}), (\ref{mn3}) that $dt_1/d|\beta|< 0,$ because
$0< t_1 < 2\pi/\sqrt{\beta^2-1}$ if $1< |\beta|< 3/\sqrt{5}.$ Thus theorem is proved.
\end{proof}

\begin{remark}
Canonical projection $p: L^2_1\rightarrow L^2$ is principal bundle with fiber  $SO(2).$ Since $L^2$ is diffeomorphic to $\mathbb{R}^2,$ there exists smooth unit vector field on $L^2$. Therefore the fibration $p$ is trivial and $L^2_1$,
$SO_0(2,1)$ are diffeomorphic to $\mathbb{R}^2\times S^1.$ Consequently
fundamental groups of Lie groups $SO_0(2,1)$ and $Sl(2)$ are isomorphic to 
$(\mathbb{Z},+)$. It would be interesting to find shortest arcs for locally isometric epimorphic covering by Lie groups $Sl(2)\rightarrow SO_0(2,1)$ and $\tilde{Sl}(2)\rightarrow Sl(2),$ where $\tilde{Sl}(2)$ is simply connected.
\end{remark}

\section{Cut locus and conjugate sets in $(SO_0(2,1),d)$}
\label{cut}

Unlike Riemannian manifolds, exponential map $\Exp$ and its restriction $\Exp_x$
for sub-Riemannian manifold $(M,d)$ without abnormal  geodesics (as in the case of $(SO_0(2,1),d)$) are defined not on $TM$ and $T_xM$ but only on $D$ and $D(x)$,
where $D$ is distribution on $M,$ taking part in definition of $d.$ Otherwise
cut locus and conjugate sets for such sub-Riemannian manifold are defined in the same way as for Riemannian one \cite{GKM}.

\begin{definition}
\label{cutl}
Cut locus $C(x)$ (respectively conjugate set $S(x)$) for a point $x$ in sub-Riemannian manifolds $M$ (without abnormal geodesics) is the set of ends
of all noncontinuable beyond its ends shortest arcs starting at the point $x$ (respectively, image of the set of critical points of the map $\Exp_x$
with respect to $\Exp_x$).
\end{definition}

\begin{theorem}
\label{clcl}
For every element $g\in (SO_0(2,1),d),$ $C(g)=gC(e)$ and $S(g)=gS(e).$ Moreover $S(g)\subset C(g),$
\begin{equation}
\label{cutle}
C(e)=\{\gamma_{\beta}(t_1(\beta)): \beta\in \mathbb{R}\},
\end{equation}
\begin{equation}
\label{conle}
S(e)=\{\gamma_{\beta}(t_1(\beta)): \beta^2\geq 9/5\}= SO(2)\smallsetminus \{e\};
\end{equation}
$S(e)$ is diffeomorphic to $\mathbb{R};$
\begin{equation}
\label{os}
\overline{S(e)}=S(e)\cup \{e\}=SO(2),
\end{equation}
$\overline{S(e)}$ is diffeomorphic to circle $S^1$;
\begin{equation}
\label{ocs}
\overline{C(e)\smallsetminus S(e)}= (C(e)\smallsetminus S(e))\cup \left\{\gamma_{\beta}(t_1(\beta))=
\gamma_{-\beta}(t_1(-\beta)): \beta = \frac{3}{\sqrt{5}}\right\},
\end{equation}
$\overline{C(e)\smallsetminus S(e)}$ is diffeomorphic to real projective plane without a point
$RP^2\smallsetminus \{\infty\};$  $C(e)$ is homeomorphic to
$(RP^2\smallsetminus \{\infty\})\cup \mathbb{R},$ where $(RP^2\smallsetminus \{\infty\})\cap \mathbb{R}$ is one-point set;
$\overline{C(e)}$ is homeomorphic to $(RP^2\smallsetminus \{\infty\})\cup S^1,$
where $(RP^2\smallsetminus \{\infty\})\cap S^1$ is one-point set.
\end{theorem}

\begin{proof}
First statement is a corollary of left invariance of the metric $d$ on $SO_0(2,1).$ Inclusion $S(g)\subset C(g),$ formulae
(\ref{cutle}), (\ref{conle}), equality in brace from (\ref{ocs}), and diffeomorphism $S(e)\cong \mathbb{R}$ are corollaries from the proof of 
proposition~\ref{zer1} and remark \ref{Ad}. Formula (\ref{os}) and diffeomorphism
$\overline{S(e)}\cong S^1$ follow from formula (\ref{conle}). Equality (\ref{ocs}) follows from formulae (\ref{cutle}), (\ref{conle}); 
$\overline{C(e)\smallsetminus S(e)}\cong RP^2\smallsetminus \{\infty\}$ follows from equalities $\gamma_{(\beta,\phi_0)}(t_1(\beta))=\gamma_{(-\beta,-\beta t_1+\phi_0 + \pi)}(t_1(-\beta))$ for $|\beta|\leq 3/\sqrt{5}.$
Now it is not difficult to prove remaining statements.
\end{proof}

\begin{remark}
It follows from (\ref{os}) and equalities $C(g)=gC(e),$ $S(g)=gS(e)$ that 
$g\in gSO(2)=\overline{S(g)}\subset \overline{C(g)}$ for all $g\in SO_0(2,1)$ 
while $x\notin \overline{C(x)}$ and $x\notin \overline{S(x)}$ for any point 
$x$ of arbitrary smooth Riemannian manifold. This constitutes radical difference
of Riemannian and sub-Riemannian manifolds.
\end{remark}

\end{document}